\documentclass[12pt]{amsart}

\usepackage{latexsym}
\usepackage{psfrag}
\usepackage{amsmath}
\usepackage{amssymb}
\usepackage{epsfig}
\usepackage{amsfonts}
\usepackage{amscd}
\usepackage{mathrsfs}
\usepackage{graphicx}
\usepackage{enumerate}
\usepackage[autostyle=false, style=english]{csquotes}
\MakeOuterQuote{"}
\usepackage{ragged2e}
\usepackage[all]{xy}
\usepackage{mathtools}
\newlength\ubwidth

\usepackage[dvipsnames]{xcolor}

\usepackage{placeins}
\usepackage{tikz}
\usetikzlibrary{shapes, positioning}
\usetikzlibrary{matrix,shapes.geometric,calc,backgrounds}

\newtheorem{theorem}{Theorem}%[section]
\newtheorem{lemma}[theorem]{Lemma}

\newtheorem{ex}[theorem]{Example}

\newtheorem{defn}[theorem]{Definition}

\usepackage{amsmath,amssymb,amsthm}
\usepackage[dvipsnames]{xcolor}
\usepackage[colorlinks=true,citecolor=RoyalBlue,linkcolor=red,breaklinks=true]{hyperref}

\usepackage{graphicx}
\usepackage{mathdots} % for anti diagonal dots
\usepackage[margin=2.5cm]{geometry}
\usepackage{ytableau}
\newcommand*\circled[1]{\tikz[baseline=(char.base)]{
		\node[shape=circle,draw,inner sep=2pt] (char) {#1};}}

\newcommand{\mathcolorbox}[2]{%
	\colorbox{#1}{$\displaystyle#2$}}

\begin{document}%%%%%%%%%%%%%%%%%%%%%%%%%%%%%%%%%%%%%%%%%%%%%%%%%%%%%%%%
	%%%%%%%%%%%%%%%%%%%%%%%%%%%%%%%%%%%%%%%%%%%%%%%%%%%%%%%%%%%%%%%%%%%%%%%%

	\title[Construction of cylindric partitions with small profiles]
	{Combinatorial constructions of generating functions 
		of cylindric partitions with small profiles into unrestricted or distinct parts}
	
	\author[Kur\c{s}ung\"{o}z]{Ka\u{g}an Kur\c{s}ung\"{o}z}
	\address{Ka\u{g}an Kur\c{s}ung\"{o}z, Faculty of Engineering and Natural Sciences, 
		Sabanc{\i} University, Tuzla, Istanbul 34956, Turkey}
	\email{kursungoz@sabanciuniv.edu}
	
	\author[\"{O}mr\"{u}uzun Seyrek]{Hal\.{ı}me \"{O}mr\"{u}uzun Seyrek}
	\address{Hal\.{ı}me \"{O}mr\"{u}uzun Seyrek, Faculty of Engineering and Natural Sciences, 
		Sabanc{\i} University, Tuzla, Istanbul 34956, Turkey}
	\email{halimeomruuzun@alumni.sabanciuniv.edu}
	
	
	
	\date{2022}
	
	\begin{abstract}
		In this paper, cylindric partitions into profiles $c=(1,1)$ and $c=(2,0)$ are considered.  
		The generating functions into unrestricted cylindric partitions 
		and cylindric partitions into distinct parts with these profiles are constructed.  
		The constructions are combinatorial and they connect 
		the cylindric partitions with ordinary partitions.  
% 		The generating function of cylindric partitions with the said profiles 
% 		turn out to be combinations of two infinite products.  
	\end{abstract}
	
	\maketitle

	\section{Introduction}
	Cylindric partitions were introduced by Gessel and Krattenthaler \cite{GesselKrattenthaler}. 
	
	\begin{defn}\label{def:cylin} Let $k$ and $\ell$ be positive integers. Let $c=(c_1,c_2,\dots, c_k)$ be a composition, where $c_1+c_2+\dots+c_k=\ell$. A \emph{cylindric partition with profile $c$} is a vector partition $\Lambda = (\lambda^{(1)},\lambda^{(2)},\dots,\lambda^{(k)})$, where each $\lambda^{(i)} = \lambda^{(i)}_1+\lambda^{(i)}_2 + \cdots +\lambda^{(i)}_{s_i}$ is a partition, such that for all $i$ and $j$,
		$$\lambda^{(i)}_j\geq \lambda^{(i+1)}_{j+c_{i+1}} \quad \text{and} \quad \lambda^{(k)}_{j}\geq\lambda^{(1)}_{j+c_1}.$$
	\end{defn}
	
	For example, the sequence $\Lambda=((6,5,4,4),(8,8,5,3),(7,6,4,2))$
	is a cylindric partition with profile $(1,2,0)$. 
	One can check that for all $j$, $\lambda^{(1)}_j\ge \lambda^{(2)}_{j+2}$, $\lambda^{(2)}_j\ge \lambda^{(3)}_{j}$
	and $\lambda^{(3)}_j\ge \lambda^{(1)}_{j+1}$.
	We can visualize the required inequalities by writing the partitions 
	in subsequent rows repeating the first row below the last one, 
	and shifting the rows below as much as necessary to the left.  
	Thus, the inequalities become the weakly decreasing of the parts to the right in each row, 
	and downward in each column.  
	\[
	\begin{array}{ccc ccc ccc}
		& & & 6 & 5 & 4 & 4\\
		& 8 & 8 & 5 & 3 & \\
		& 7 & 6 & 4 & 2& \\
		\textcolor{lightgray}{6} & \textcolor{lightgray}{5} 
		& \textcolor{lightgray}{4} & \textcolor{lightgray}{4}	
	\end{array}
	\]
	The repeated first row is shown in gray.  
	
	The size $|\Lambda|$ of a cylindric partition $\Lambda = (\lambda^{(1)},\lambda^{(2)},\dots,\lambda^{(k)})$ is defined to be the sum of all the parts in the partitions $\lambda^{(1)},\lambda^{(2)},\dots,\lambda^{(k)}$. The largest part of a cylindric partition $\Lambda$ is defined to be the maximum part among all the partitions in $\Lambda$, and it is denoted by $\max(\Lambda)$. 
	The following generating function
	$$F_c(z,q):=\sum_{\Lambda\in \mathcal{P}_c} z^{\max{(\Lambda)}}q^{|\Lambda |}$$
	is the generating function for cylindric partitions, where $\mathcal{P}_c$ denotes the set of all cylindric partitions with profile $c$. 
	
	In 2007, Borodin \cite{Borodin} showed that when one sets $z=1$ to this generating function, it turns out to be a very nice infinite product. 
	
	\begin{theorem}[Borodin, 2007]	\label{theorem-Borodin}
		Let $k$ and $\ell$ be positive integers, and let $c=(c_1,c_2,\dots,c_k)$ be a composition of $\ell$. Define $t:=k+\ell$ and $s(i,j) := c_i+c_{i+1}+\dots+ c_j$. Then,
		\begin{equation} 
			\label{BorodinProd}
			F_c(1,q) = \frac{1}{(q^t;q^t)_\infty} \prod_{i=1}^k \prod_{j=i}^k \prod_{m=1}^{c_i} \frac{1}{(q^{m+j-i+s(i+1,j)};q^t)_\infty} \prod_{i=2}^k \prod_{j=2}^i \prod_{m=1}^{c_i} \frac{1}{(q^{t-m+j-i-s(j,i-1)};q^t)_\infty}.
		\end{equation}
	\end{theorem}
	
	The identity (\refeq{BorodinProd}) is a very strong tool to find product representation of generating functions of cylindric partitions with a given profile explicitly.  
	
	Cylindric partitions have been studied intensively since their 
	introduction \cite{GesselKrattenthaler}.  
	Prominent examples are constructing Andrews-Gordon~\cite{Andrews-PNAS} type 
	evidently positive multiple series companions 
	to some cases in Borodin's theorem~\cite{CDU, CW, OW},  
	or even connections with theoretical physics~\cite{IKS}.

	The purpose of this paper is to construct generating functions 
	of cylindric partitions with small profiles into unrestricted or distinct parts. 
	In Section 2, we combinatorially reprove generating functions 
	for cylindric partitions with profiles $c=(1,1)$ and $c=(2,0)$. 
	The construction is based on the fact that if we have 
	a cylindric partition with profile $c=(1,1)$ or $c=(2,0)$, 
	then it can be decomposed into a pair of partitions $(\mu,\beta)$ 
	by a series of combinatorial moves. 
	The results in Section \ref{secGenFuncsUnrestricted} are limiting cases, 
	therefore corollaries, of~\cite[eq. (7.25)]{Warnaar}.  
	The proof techniques are different, though.  
	The approach in Section \ref{secGenFuncsUnrestricted} 
	seems to apply in~\cite[eq. (7.25)]{Warnaar}
	for $k = 1$ and $s=$ 1 or 2.
	% Then, using the one-to-one correspondence between the cylindric partitions 
	% and the pair of partitions $(\mu,\beta)$ as a key, 
	% the generating functions follow immediately. 
	In Section 3, we consider cylindric partitions with small profiles into distinct parts. 
	We construct generating functions for such partitions with profiles $c=(1,1)$ or $c=(2,0)$, 
	which turn out to be combinations of infinite products. We refer the reader to \cite{BU}, where cylindric partitions into distinct parts are also studied. 
	We conclude by constructing an evidently positive series generating function 
	for cylindric partitions with small profiles into odd parts.  
	
	\section{Generating Functions of Cylindric Partitions With Profiles \\ $c=(1,1)$ and $c=(2,0)$}
	\label{secGenFuncsUnrestricted}
	By using (\ref{BorodinProd}), one can easily show that 
	
	\begin{equation} \label{c=(1,1)}
		F_c(1,q) = \frac{(-q;q^2)_\infty}{(q;q)_\infty},
	\end{equation}
	where $c=(1,1)$.
	
	In the following theorem, we will give a combinatorial proof of identity (\ref{c=(1,1)}).
	
	\begin{theorem} 
		\label{Fc(1,q) when c=(1,1)}
		Let $c=(1,1)$. Then the generating function of cylindric partitions with profile $c$ is given by 
		
		\begin{equation*} 
			F_c(1,q) = \frac{(-q;q^2)_\infty}{(q;q)_\infty}.
		\end{equation*}
	\end{theorem}
	
	\begin{proof}
		We will show that each cylindric partition $\lambda$ with profile $c=(1,1)$ corresponds 
		to a unique pair of partitions $(\mu,\beta)$, 
		where $\mu$ is an ordinary partition and $\beta$ is a partition with distinct odd parts. 
		Conversely, we will show that each pair of partitions $(\mu,\beta)$ will correspond 
		to a unique cylindric partition with profile $c=(1,1)$, 
		where $\mu$ is an ordinary partition and $\beta$ is a partition into distinct odd parts. 
		In this way, we will get the desired generating function for cylindric partitions with profile $c=(1,1)$.
		\begin{align}
			\nonumber
			F_c(1,q) = \sum_{\lambda} q^{\vert \lambda \vert}
			= \sum_{(\mu, \beta)} q^{\vert \mu \vert + \vert \beta \vert}
			= \left( \sum_{\mu} q^{\vert \mu \vert} \right) \, \left( \sum_{\beta} q^{\vert \beta \vert} \right)
			= \frac{1}{ (q; q)_\infty } \, (-q; q^2)_\infty, 
		\end{align}
		where $\lambda$, $\mu$, and $\beta$ are as described above. The first identity is the definition of $F_c(1,q)$.
		The second identity will be proven below.  
		The third follows from the fact that $\mu$ and $\beta$ are independent, 
		and the last one because unrestricted partitions and partitions into distinct odd parts 
		have the displayed infinite product generating functions~\cite{TheBlueBook}.

		Let  $\lambda$ be a cylindric partition  with profile $c=(1,1)$. Then $\lambda$ has the following form:
		
		% \[
		% \begin{array}{ccc ccc ccc}
			% & & & a_1 & a_2 & a_3 & \ldots &a_{r-1} & a_r \\
			% & & b_1 & b_2 & b_3 & \ldots  & b_s
			% \end{array}
		% \]
		% where $r-1 \leq s \leq r+1$.
		
		\begin{align}
			\nonumber
			\begin{array}{ccc ccc ccc}
				& & a_1 & a_2 & a_3 & \ldots &a_{r-1} & a_r \\
				& b_1 & b_2 & b_3 & \ldots  & b_s \\ 
				\textcolor{lightgray}{a_1} & \textcolor{lightgray}{a_2} & \textcolor{lightgray}{a_3} 
				& \textcolor{lightgray}{\ldots} & \textcolor{lightgray}{a_r} & 
			\end{array}, 
		\end{align}
		where $r-1 \leq s \leq r+1$.  
		The last line is a repetition of the first one, 
		and the parts are weakly decreasing from left to right and downward.

		If we allow zeros at the end of partitions, we can take $s = r$.  
		Namely, if $s = r-1$, then we append $b_r = 0$; 
		and if $s = r+1$, then we both append $a_{r+1} = 0$ and update $r+1$ as $r$.  
		So, without loss of generality, our cylindric partition with profile $c=(1,1)$ looks like
		\begin{align}
			\nonumber
			\lambda = \begin{array}{ccc ccc ccc}
				& & a_1 & a_2 & a_3 & \ldots & a_{r-1} & a_r \\
				& b_1 & b_2 & b_3 & \ldots & b_{r-1} & b_r \\ 
				\textcolor{lightgray}{a_1} & \textcolor{lightgray}{a_2} & \textcolor{lightgray}{a_3} 
				& \textcolor{lightgray}{\ldots} & \textcolor{lightgray}{a_{r-1}} & \textcolor{lightgray}{a_r} & 
			\end{array}.  
		\end{align}
		At this point, only $a_r$ or $b_r$ may be zero, but not both.  
		Therefore either all parts or all parts but one in $\lambda$ are positive.  
		During the process of obtaining $\mu$ and $\beta$, 
		some or all parts of $\lambda$ may become zero.  
		It is possible that $\mu$ is a partition consisting entirely of zeros, 
		i.e., the empty partition.  
		But that does not create a problem because $r$ is determined at the beginning, and it is fixed.  
		% 	% KK
		% 		{\tiny \textcolor{red}{(KK: I'm not sure if we really need this explanation).  } }
		
		Our goal is to transform $\lambda$ into another cylindric partition $\widetilde{\lambda}$ of the same profile 
		\begin{align}
			\label{cylPtnLambdaTilde}
			\widetilde{\lambda} = \begin{array}{ccc ccc ccc}
				& & \widetilde{a}_1 & \widetilde{a}_2 & \widetilde{a}_3 & \ldots & \widetilde{a}_{r-1} & \widetilde{a}_r \\
				& \widetilde{b}_1 & \widetilde{b}_2 & \widetilde{b}_3 & \ldots & \widetilde{b}_{r-1} & \widetilde{b}_r \\ 
				\textcolor{lightgray}{\widetilde{a}_1} & \textcolor{lightgray}{\widetilde{a}_2} & \textcolor{lightgray}{\widetilde{a}_3} 
				& \textcolor{lightgray}{\ldots} & \textcolor{lightgray}{\widetilde{a}_{r-1}} & \textcolor{lightgray}{\widetilde{a}_r} & 
			\end{array}
		\end{align}
		with the additional property that $\widetilde{b}_j \geq \widetilde{a}_j$ for all $j = 1, 2, \ldots, r$, 
		allowing zeros at the end.  
		Then, parts of $\widetilde{\lambda}$ can be listed as
		\begin{align}
			\nonumber
			\mu = ( \widetilde{b}_1, \widetilde{a}_1, \widetilde{b}_2, \widetilde{a}_2, \ldots, 
			\widetilde{b}_r, \widetilde{a}_r )
		\end{align}
		to obtain the promised unrestricted partition $\mu$.  
		The remaining inequalities $ \widetilde{a}_j \geq \widetilde{b}_{j+1} $ for $j = 1, 2, \ldots, (r-1)$ 
		are ensured by the fact that $\widetilde{\lambda}$ is a cylindric partition with profile $c=(1,1)$.  
		
		We will do this by a series of transformations on $\lambda$ 
		which will be recorded as a partition $\beta$ into distinct odd parts.  
		We will then argue that $ \vert \lambda \vert = \vert \mu \vert + \vert \beta \vert$.
		
		We read the parts of the cylindric partition 
		
		\begin{align}
			\nonumber
			\lambda =	\begin{array}{ccc ccc ccc}
				& & a_1 & a_2 & a_3 & \ldots & a_{r-1} & a_r \\
				& b_1 & b_2 & b_3 & \ldots & b_{r-1} & b_r \\ 
				\textcolor{lightgray}{a_1} & \textcolor{lightgray}{a_2} & \textcolor{lightgray}{a_3} 
				& \textcolor{lightgray}{\ldots} & \textcolor{lightgray}{a_{r-1}} & \textcolor{lightgray}{a_r} & 
			\end{array}  
		\end{align}
		as the pairs: $[b_1,a_1], [b_2,a_2], [b_3,a_3], \ldots, [b_r,a_r]$. We start with the rightmost pair $[b_r,a_r]$. If $b_r \geq a_r$, there's nothing to do.  
		We simply set $\widetilde{b}_r = b_r$, $\widetilde{a}_r = a_r$, and do not add any parts to $\beta$ yet.  
		
		If $b_r < a_r$, then we 
		\begin{itemize}
			\item switch places of $a_r$ and $b_r$, 
			\item subtract 1 from each of the parts $a_1$, $a_2$, \ldots, $a_r$, $b_1$, $b_2$, \ldots $b_{r-1}$, 
			\item set $\widetilde{b}_r = a_{r}-1$ and $\widetilde{a}_r = b_r$, 
			\item add the part $(2r-1)$ to $\beta$.  
		\end{itemize}
		We need to perform several checks here.  
		First, we will show that at each of the steps listed above, 
		the intermediate cylindric partition satisfies the weakly decreasing condition
		across rows and down columns.  
		The affected parts are highlighted.  
		\begin{align}
			\nonumber
			\begin{array}{ccc ccc ccc}
				& & a_1 & a_2 & a_3 & \ldots & \mathcolorbox{yellow!50}{a_{r-1}} & \mathcolorbox{yellow!50}{a_r} \\
				& b_1 & b_2 & b_3 & \ldots & \mathcolorbox{yellow!50}{b_{r-1}} & \mathcolorbox{yellow!50}{b_r} \\ 
				\textcolor{lightgray}{a_1} & \textcolor{lightgray}{a_2} & \textcolor{lightgray}{a_3} 
				& \textcolor{lightgray}{\ldots} 
				& \mathcolorbox{yellow!25}{\textcolor{lightgray}{a_{r-1}}} 
				& \mathcolorbox{yellow!25}{\textcolor{lightgray}{a_r}} & 
			\end{array}
		\end{align}
		\begin{align*}
			\Bigg\downarrow \textrm{ after switching places of } a_r \textrm{ and } b_r
		\end{align*}
		\begin{align}
			\nonumber
			\begin{array}{ccc ccc ccc}
				& & a_1 & a_2 & a_3 & \ldots & \mathcolorbox{yellow!50}{a_{r-1}} & \mathcolorbox{yellow!50}{b_r} \\
				& b_1 & b_2 & b_3 & \ldots & \mathcolorbox{yellow!50}{b_{r-1}} & \mathcolorbox{yellow!50}{a_r} \\ 
				\textcolor{lightgray}{a_1} & \textcolor{lightgray}{a_2} & \textcolor{lightgray}{a_3} 
				& \textcolor{lightgray}{\ldots} 
				& \mathcolorbox{yellow!25}{\textcolor{lightgray}{a_{r-1}}} 
				& \mathcolorbox{yellow!25}{\textcolor{lightgray}{b_r}} & 
			\end{array}
		\end{align}
		The inequalities $a_{r-1} \geq b_r$ and $a_{r-1} \geq a_r$ carry over from the original cylindric partition.  
		The inequalities $b_{r-1} \geq a_r$ and $b_{r-1} \geq b_r$
		are also two of the inequalities implied by the original cylindric partition.  
		All other inequalities are untouched.  
		At this point, we have not altered the weight of the cylindric partition yet.  
		\begin{align*}
			\Bigg\downarrow \textrm{ after subtracting 1 from the listed parts }
		\end{align*}
		\begin{align}
			\nonumber
			\begin{array}{ccc ccc ccc}
				& & (a_1 - 1) & (a_2 - 1) & (a_3 - 1) & \ldots & \mathcolorbox{yellow!50}{(a_{r-1} - 1)} & \mathcolorbox{yellow!50}{b_r} \\
				& (b_1 - 1) & (b_2 - 1) & (b_3 - 1) & \ldots & \mathcolorbox{yellow!50}{(b_{r-1} - 1)} 
				& \mathcolorbox{yellow!50}{(a_r - 1)} \\ 
				\textcolor{lightgray}{(a_1 - 1)} & \textcolor{lightgray}{(a_2 - 1)} 
				& \textcolor{lightgray}{(a_3 - 1)} 
				& \textcolor{lightgray}{\ldots} 
				& \mathcolorbox{yellow!25}{\textcolor{lightgray}{(a_{r-1} - 1)}} 
				& \mathcolorbox{yellow!25}{\textcolor{lightgray}{b_r}} & 
			\end{array}
		\end{align}
		We argue that this is still a valid cylindric partition.  
		The only inequalities that need to be verified are 
		$a_{r-1} - 1 \geq b_r$ and $b_{r-1} - 1 \geq b_r$.  
		Because of the original cylindric partition, we have 
		$a_{r-1} \geq a_r$ and $b_{r-1} \geq a_r$.  
		Because of the case we are examining $a_r > b_r$, 
		so that $a_r - 1 \geq b_r$, both being integers.  
		Combining $a_{r-1} - 1 \geq a_r -1$, $b_{r-1} - 1 \geq a_r - 1$ and $a_r - 1 \geq b_r$ 
		yield the desired inequalities.  
		\begin{align*}
			\Bigg\downarrow \textrm{ after relabeling }
		\end{align*}
		\begin{align}
			\nonumber
			\begin{array}{ccc ccc ccc}
				& & (a_1 - 1) & (a_2 - 1) & (a_3 - 1) & \ldots & {(a_{r-1} - 1)} & {\widetilde{a}_r} \\
				& (b_1 - 1) & (b_2 - 1) & (b_3 - 1) & \ldots & {(b_{r-1} - 1)} 
				& {\widetilde{b}_r} \\ 
				\textcolor{lightgray}{(a_1 - 1)} & \textcolor{lightgray}{(a_2 - 1)} & \textcolor{lightgray}{(a_3 - 1)} 
				& \textcolor{lightgray}{\ldots} & \textcolor{lightgray}{(a_{r-1} - 1)} 
				& {\textcolor{lightgray}{\widetilde{a}_r}} & 
			\end{array}
		\end{align}
		Now we have ${\widetilde{b}_r} \geq {\widetilde{a}_r}$ since $a_r - 1 \geq b_r$.  
		Also, we subtracted 1 from exactly $2r-1$ parts.  
		We add this $(2r-1)$ as a part in $\beta$.  
		At the beginning, $\beta$ was the empty partition, 
		so it is a partition into distinct odd parts both before and after this transformation.  
		The sum of the weight of $\beta$ and the weight of the cylindric partition remains constant.  
		It is possible that either or both $\widetilde{a}_r$ and $\widetilde{b}_r$ may be zero, 
		along with some other parts.  
		For example, in the extreme case that $a_1 = a_2 = \cdots = a_r = 1$, 
		$b_1 = b_2 = \cdots = b_{r-1} = 1$ and $b_r = 0$, 
		the cylindric partition becomes the empty partition after the transformation we illustrated.  
		
		We should mention that after this point 
		there is no harm in renaming $(a_i - 1)$'s $a_i$'s and $(b_i - 1)$'s $b_i$'s, where applicable.  
		This will lead to the cleaner exposition down below.  
		There is no loss of information, since the subtracted 1's are recorded 
		as a part in $\beta$ already.  
		
		Then, we repeat the following process for $j = (r-1), (r-2), \ldots, 2, 1$ in the given order.  
		At the beginning of the $j$th step, we have the intermediate cylindric partition 
		\begin{align}
			\nonumber
			\begin{array}{ccc ccc ccc c}
				& & a_1 & a_2 & \cdots & a_{j-1} & a_j & \widetilde{a}_{j+1} & \cdots & \widetilde{a}_r \\
				& b_1 & b_2 & \cdots & b_{j-1} & b_j & \widetilde{b}_{j+1} & \cdots & \widetilde{b}_r & \\ 
				\textcolor{lightgray}{a_1} & \textcolor{lightgray}{a_2} & \textcolor{lightgray}{\cdots} & 
				\textcolor{lightgray}{a_{j-1}} & \textcolor{lightgray}{a_j} & \textcolor{lightgray}{\widetilde{a}_{j+1}} & 
				\textcolor{lightgray}{\cdots} & \textcolor{lightgray}{\widetilde{a}_r} & & 
			\end{array}.  
		\end{align}
		The parts weakly decrease from left to right and downward, 
		and the third line is a repetition of the first one.  
		This intermediate cylindric partition satisfies the additional inequalities 
		\begin{align}
			\nonumber
			\widetilde{b}_{j+1} \geq \widetilde{a}_{j+1}, \quad 
			\widetilde{b}_{j+2} \geq \widetilde{a}_{j+2}, \quad 
			\cdots \quad 
			\widetilde{b}_{r} \geq \widetilde{a}_{r}.  
		\end{align}
		Some or all parts in this intermediate partition may be zero.  
		
		We focus on the $j$th pair $[b_j, a_j]$.  
		If $b_j \geq a_j$ already, then we do not alter either the intermediate cylindric partition 
		or the partition $\beta$ into distinct odd parts.  
		We just relabel $b_j$ as $\widetilde{b}_j$, $a_j$ as $\widetilde{a}_j$, 
		and move on to the $(j-1)$th pair.  
		
		In the other case $a_j > b_j$, we
		\begin{itemize}
			\item switch places of $a_j$ and $b_j$, 
			\item subtract 1 from each of the parts $a_1$, $a_2$, \ldots, $a_j$, $b_1$, $b_2$, \ldots $b_{j-1}$, 
			\item set $\widetilde{b}_j = a_{j}-1$ and $\widetilde{a}_j = b_j$, 
			\item add the part $(2j-1)$ to $\beta$.  
		\end{itemize}
		We again perform several checks as in the $r$th case, but this time there are inequalities 
		that involve parts that lie to the right of $a_j$ and $b_j$.  
		We first show that the listed operations do not violate the 
		weakly decreasing condition on the cylindric partition across rows and down columns.  
		The affected parts are highlighted.  
% 		\begin{align}
% 			\nonumber
% 			\begin{array}{ccc ccc ccc c}
% 				& & a_1 & a_2 & \cdots & \mathcolorbox{yellow!50}{a_{j-1}} & 
% 				\mathcolorbox{yellow!50}{a_j} & \mathcolorbox{yellow!50}{\widetilde{a}_{j+1}} & \cdots & \widetilde{a}_r \\
% 				& b_1 & b_2 & \cdots & \mathcolorbox{yellow!50}{b_{j-1}} & \mathcolorbox{yellow!50}{b_j} & 
% 				\mathcolorbox{yellow!50}{\widetilde{b}_{j+1}} & \cdots & \widetilde{b}_r & \\ 
% 				\textcolor{lightgray}{a_1} & \textcolor{lightgray}{a_2} & \textcolor{lightgray}{\cdots} & 
% 				\mathcolorbox{yellow!25}{\textcolor{lightgray}{a_{j-1}}} & \mathcolorbox{yellow!25}{\textcolor{lightgray}{a_j}} & 
% 				\mathcolorbox{yellow!25}{\textcolor{lightgray}{\widetilde{a}_{j+1}}} & 
% 				\textcolor{lightgray}{\cdots} & \textcolor{lightgray}{\widetilde{a}_r} & & 
% 			\end{array} 
% 		\end{align}
% 		\begin{align*}
% 			\Bigg\downarrow \textrm{ after switching places of } a_j \textrm{ and } b_j
% 		\end{align*}
        We switch the places of $a_j$ and $b_j$ to obtain
		\begin{align}
			\nonumber
			\begin{array}{ccc ccc ccc c}
				& & a_1 & a_2 & \cdots & \mathcolorbox{yellow!50}{a_{j-1}} & 
				\mathcolorbox{yellow!50}{b_j} & \mathcolorbox{yellow!50}{\widetilde{a}_{j+1}} & \cdots & \widetilde{a}_r \\
				& b_1 & b_2 & \cdots & \mathcolorbox{yellow!50}{b_{j-1}} & \mathcolorbox{yellow!50}{a_j} & 
				\mathcolorbox{yellow!50}{\widetilde{b}_{j+1}} & \cdots & \widetilde{b}_r & \\ 
				\textcolor{lightgray}{a_1} & \textcolor{lightgray}{a_2} & \textcolor{lightgray}{\cdots} & 
				\mathcolorbox{yellow!25}{\textcolor{lightgray}{a_{j-1}}} & 
				\mathcolorbox{yellow!25}{\textcolor{lightgray}{b_j}} & 
				\mathcolorbox{yellow!25}{\textcolor{lightgray}{\widetilde{a}_{j+1}}} & 
				\textcolor{lightgray}{\cdots} & \textcolor{lightgray}{\widetilde{a}_r} & & 
			\end{array}.  
		\end{align}
		Each of the required inequalities 
		\begin{align}
			\nonumber 
			a_{j-1} \geq b_j \geq \widetilde{a}_{j+1}, \quad 
			b_{j-1} \geq a_j \geq \widetilde{b}_{j+1}, \quad 
			b_{j-1} \geq b_j, \quad 
			a_{j-1} \geq a_j \geq \widetilde{a}_{j+1}, \quad 
			\textrm{ and } \quad 
			b_{j} \geq \widetilde{b}_{j+1}
		\end{align}
		are already implied in the cylindric partition before the change.  
		The inequalities between the non-highlighted parts carry over.  
% 		\begin{align*}
% 			\Bigg\downarrow \textrm{ after subtracting 1 from each of the listed parts } 
% 		\end{align*}
% 		\begin{align}
% 			\nonumber
% 			\begin{array}{ccc ccc ccc c}
% 				& & (a_1 - 1) & (a_2 - 1) & \cdots & \mathcolorbox{yellow!50}{(a_{j-1} - 1)} & 
% 				\mathcolorbox{yellow!50}{b_j} & \mathcolorbox{yellow!50}{\widetilde{a}_{j+1}} & \cdots & \widetilde{a}_r \\
% 				& (b_1 - 1) & (b_2 - 1) & \cdots & \mathcolorbox{yellow!50}{(b_{j-1} - 1)} & \mathcolorbox{yellow!50}{(a_j - 1)} & 
% 				\mathcolorbox{yellow!50}{\widetilde{b}_{j+1}} & \cdots & \widetilde{b}_r & \\ 
% 				\textcolor{lightgray}{(a_1 - 1)} & \textcolor{lightgray}{(a_2 - 1)} & \textcolor{lightgray}{\cdots} & 
% 				\mathcolorbox{yellow!25}{\textcolor{lightgray}{(a_{j-1} - 1)}} & 
% 				\mathcolorbox{yellow!25}{\textcolor{lightgray}{b_j}} & 
% 				\mathcolorbox{yellow!25}{\textcolor{lightgray}{\widetilde{a}_{j+1}}} & 
% 				\textcolor{lightgray}{\cdots} & \textcolor{lightgray}{\widetilde{a}_r} & & 
% 			\end{array} 
% 		\end{align}
        We then subtract one from each of the listed parts.  
		The inequalities we need to verify are 
		\begin{align}
			\nonumber 
			a_{j-1} - 1 \geq b_j, \quad 
			a_j - 1 \geq \widetilde{b}_{j+1}, \quad 
			b_{j-1}-1 \geq b_j, \quad 
			\textrm{ and } \quad 
			a_j - 1 \geq \widetilde{a}_{j+1}.  
		\end{align}
		By the cylindric partition two steps ago, we have 
		\begin{align}
			\nonumber 
			a_{j-1} \geq a_j, \quad 
			b_j \geq \widetilde{b}_{j+1}, \quad 
			b_{j-1} \geq a_j, 
			\textrm{ and } \quad 
			b_j \geq \widetilde{a}_{j+1}.  
		\end{align}
		By the hypothesis, $a_j > b_j$, so $a_j-1 \geq b_j$.  
		This last inequality, combined with the last four displayed inequalities 
		yield the inequalities we wanted.  
% 		\begin{align*}
% 			\Bigg\downarrow \textrm{ after relabeling } 
% 		\end{align*}
% 		\begin{align}
% 			\nonumber
% 			\begin{array}{ccc ccc ccc c}
% 				& & (a_1 - 1) & (a_2 - 1) & \cdots & {(a_{j-1} - 1)} & 
% 				{\widetilde{a}_j} & {\widetilde{a}_{j+1}} & \cdots & \widetilde{a}_r \\
% 				& (b_1 - 1) & (b_2 - 1) & \cdots & {(b_{j-1} - 1)} & { \widetilde{b}_j } & 
% 				{\widetilde{b}_{j+1}} & \cdots & \widetilde{b}_r & \\ 
% 				\textcolor{lightgray}{(a_1 - 1)} & \textcolor{lightgray}{(a_2 - 1)} & \textcolor{lightgray}{\cdots} & 
% 				{\textcolor{lightgray}{(a_{j-1} - 1)}} & 
% 				{\textcolor{lightgray}{\widetilde{a}_j}} & 
% 				{\textcolor{lightgray}{\widetilde{a}_{j+1}}} & 
% 				\textcolor{lightgray}{\cdots} & \textcolor{lightgray}{\widetilde{a}_r} & & 
% 			\end{array} 
% 		\end{align}
        Then we relabel $b_j$ as $\widetilde{a}_j$ 
        and $a_j$ as $\widetilde{b}_j$ in their respective new places.  
		We have $\widetilde{b}_j \geq \widetilde{a}_j$, since $a_j-1 \geq b_j$.  
		
		On the other hand, we subtracted a total of $(2j-1)$ 1's from the parts of the intermediate cylindric partition, 
		and now we add $(2j-1)$ to $\beta$.  
		$\beta$ still has distinct odd parts, because the smallest part we had added to $\beta$ must be $\geq (2j+1)$ 
		in the previous step.  It is also possible that $\beta$ was empty before adding $(2j-1)$.  
		We should note that $(2j-1)$ is the smallest part in $\beta$ at the moment.  
		In any case, we have 
		\begin{align}
			\nonumber 
			\vert \lambda \vert = \vert \beta \vert + 
			\textrm{the weight of the intermediate cylindric partition}.  
		\end{align}
		$\lambda$ is the original cylindric partition, before any changes.  
		
		Like after the $r$th step, there is no danger in renaming $(a_i - 1)$'s $a_i$'s and $(b_i - 1)$'s $b_i$'s, 
		where necessary.  
		
		Once this process is finished, we have the cylindric partition $\widetilde{\lambda}$ 
		as given in \eqref{cylPtnLambdaTilde}, 
		The nonzero parts of $\widetilde{\lambda}$ is listed as parts 
		of the unrestricted partition $\mu$, 
		the alterations in obtaining $\widetilde{\lambda}$ are recorded 
		as parts of the partition $\beta$ into distinct odd parts, 
		and one direction of the proof is over.

		Next; given $(\mu, \beta)$, 
		where $\mu$ is an unrestricted partition, 
		and $\beta$ is a partition into distinct odd parts, 
		we will produce a unique cylindric partition $\lambda$ with profile $c=(1,1)$ 
		such that 
		\begin{align}
			\nonumber 
			\vert \lambda \vert = \vert \mu \vert + \vert \beta \vert.  
		\end{align}
		The parts of $\mu$ in their respective order are relabeled as: 
		\begin{align}
			\nonumber
			\mu = & \mu_1 + \mu_2 + \cdots + \mu_l \\
			= & \widetilde{b}_1 + \widetilde{a}_1 + \cdots + \widetilde{b}_s + \widetilde{a}_s.  
		\end{align}
		The relabeling requires an even number of parts, 
		which can be solved by appending a zero at the end of $\mu$ if necessary.  
		Then, the $\widetilde{b}$'s and $\widetilde{a}$'s are arranged as 
		the cylindric partition
		\begin{align}
			\nonumber
			\widetilde{\lambda}
			= \begin{array}{ccc ccc}
				& & \widetilde{a}_1 & \widetilde{a}_2 & \cdots & \widetilde{a}_s \\ 
				& \widetilde{b}_1 & \widetilde{b}_2 & \cdots & \widetilde{b}_s & \\
				\textcolor{lightgray}{\widetilde{a}_1} & \textcolor{lightgray}{\widetilde{a}_2} & 
				\textcolor{lightgray}{\cdots} & \textcolor{lightgray}{\widetilde{a}_s} & & 
			\end{array}.  
		\end{align}
		All of the required inequalities 
		$\widetilde{a}_j \geq \widetilde{a}_{j+1}$, $\widetilde{b}_j \geq \widetilde{b}_{j+1}$, 
		$\widetilde{a}_j \geq \widetilde{b}_{j+1}$, and $\widetilde{b}_j \geq \widetilde{a}_{j+1}$
		for $j = 1, 2, \ldots, s-1$ are implied by the inequalities between parts of $\mu$.  
		$\widetilde{\lambda}$ has the additional property that 
		$\widetilde{b}_j \geq \widetilde{a}_j$ for $j = 1, 2, \ldots, s$.  
		This is the $\widetilde{\lambda}$ we obtained in the first half of the proof, 
		except for the possibly different number of zeros at the end(s).  
		The positive parts and their positions are the same.  
		
		For the smallest part $(2j-1)$ in $\beta$, we do the following.  
		\begin{itemize}
			\item delete the part from $\beta$, 
			\item add 1 to all parts $a_1, \ldots, \widetilde{a}_{j-1}, b_1, \ldots, \widetilde{b}_j$, 
			\item switch places of $\widetilde{a}_j$ and $(\widetilde{b}_j + 1)$, 
			\item rename $(a_1 + 1)$, \ldots, $(\widetilde{a}_{j-1} + 1)$, $(\widetilde{b}_j + 1)$, 
			$(b_1 + 1)$, \ldots, $(\widetilde{b}_{j-1} + 1)$, $\widetilde{a_j}$, in their respective order 
			as $a_1$, \ldots, $a_{j-1}$, $a_j$, $b_1$, \ldots, $b_{j-1}$, $b_j$.  
		\end{itemize}
		We repeat this procedure until $\beta$ becomes the empty partition, 
		at which time $\widetilde{\lambda}$ has evolved into $\lambda$, 
		the cylindric partition with profile $c=(1,1)$ we have been aiming at.  
		
		There are a few details to clarify, including the notation.  
		We start by verifying that the inequalities required by the cylindric partition 
		are satisfied at each step.  The affected parts are highlighted.  
		We start with the cylindric partition just before the transformations.  
% 		\begin{align}
% 			\label{cylPtnBackwardPhaseStart} 
% 			\begin{array}{ccc ccc ccc}
% 				& & a_1 & \cdots & \mathcolorbox{yellow!50}{\widetilde{a}_{j-1}} & 
% 				\mathcolorbox{yellow!50}{\widetilde{a}_j} & 
% 				\mathcolorbox{yellow!50}{\widetilde{a}_{j+1}} & 
% 				\ldots & \widetilde{a}_s \\
% 				& b_1 & \cdots & \mathcolorbox{yellow!50}{\widetilde{b}_{j-1}} & 
% 				\mathcolorbox{yellow!50}{\widetilde{b}_j} & 
% 				\mathcolorbox{yellow!50}{\widetilde{b}_{j+1}} & 
% 				\ldots & \widetilde{b}_s & \\
% 				\textcolor{lightgray}{a_1} & \textcolor{lightgray}{\cdots} & 
% 				\mathcolorbox{yellow!25}{ \textcolor{lightgray}{\widetilde{a}_{j-1}} } & 
% 				\mathcolorbox{yellow!25}{ \textcolor{lightgray}{\widetilde{a}_j} } & 
% 				\mathcolorbox{yellow!25}{ \textcolor{lightgray}{\widetilde{a}_{j+1}} } & 
% 				\ldots & \widetilde{a}_s & & 
% 			\end{array}
% 		\end{align}
% 		\begin{align*}
% 			\Bigg\downarrow \textrm{ after adding 1's to the listed parts } 
% 		\end{align*}
        We add one to each of the listed parts.  
% 		\begin{align}
% 			\nonumber 
% 			\begin{array}{ccc ccc ccc}
% 				& & (a_1 + 1) & \cdots & \mathcolorbox{yellow!50}{ (\widetilde{a}_{j-1} + 1) } & 
% 				\mathcolorbox{yellow!50}{\widetilde{a}_j} & 
% 				\mathcolorbox{yellow!50}{\widetilde{a}_{j+1}} & 
% 				\ldots & \widetilde{a}_s \\
% 				& (b_1 + 1) & \cdots & \mathcolorbox{yellow!50}{ (\widetilde{b}_{j-1} + 1) } & 
% 				\mathcolorbox{yellow!50}{ (\widetilde{b}_j + 1) } & 
% 				\mathcolorbox{yellow!50}{\widetilde{b}_{j+1}} & 
% 				\ldots & \widetilde{b}_s & \\
% 				\textcolor{lightgray}{ (a_1 + 1) } & \textcolor{lightgray}{\cdots} & 
% 				\mathcolorbox{yellow!25}{ \textcolor{lightgray}{ (\widetilde{a}_{j-1} + 1) } } & 
% 				\mathcolorbox{yellow!25}{ \textcolor{lightgray}{\widetilde{a}_j} } & 
% 				\mathcolorbox{yellow!25}{ \textcolor{lightgray}{\widetilde{a}_{j+1}} } & 
% 				\ldots & \widetilde{a}_s & & 
% 			\end{array}
% 		\end{align}
		The required inequalities are naturally satisfied here, 
		because the parts which are supposed to be weakly greater are increased.  
% 		\begin{align*}
% 			\Bigg\downarrow \textrm{ after switching the places of } \widetilde{a}_j 
% 			\textrm{ and } ( \widetilde{b}_j + 1 )
% 		\end{align*}
        Then, we switch places of $\widetilde{a}_j$ and $\widetilde{b}_j + 1$.  
		\begin{align}
			\nonumber 
			\begin{array}{ccc ccc ccc}
				& & (a_1 + 1) & \cdots & \mathcolorbox{yellow!50}{ (\widetilde{a}_{j-1} + 1) } & 
				\mathcolorbox{yellow!50}{ (\widetilde{b}_j + 1) } & 
				\mathcolorbox{yellow!50}{\widetilde{a}_{j+1}} & 
				\ldots & \widetilde{a}_s \\
				& (b_1 + 1) & \cdots & \mathcolorbox{yellow!50}{ (\widetilde{b}_{j-1} + 1) } & 
				\mathcolorbox{yellow!50}{ \widetilde{a}_j } & 
				\mathcolorbox{yellow!50}{\widetilde{b}_{j+1}} & 
				\ldots & \widetilde{b}_s & \\
				\textcolor{lightgray}{ (a_1 + 1) } & \textcolor{lightgray}{\cdots} & 
				\mathcolorbox{yellow!25}{ \textcolor{lightgray}{ (\widetilde{a}_{j-1} + 1) } } & 
				\mathcolorbox{yellow!25}{ \textcolor{lightgray}{ (\widetilde{b}_j + 1) } } & 
				\mathcolorbox{yellow!25}{ \textcolor{lightgray}{\widetilde{a}_{j+1}} } & 
				\ldots & \widetilde{a}_s & & 
			\end{array}
		\end{align}
		Again, the required inequalities are implied by the cylindric partition in the previous step.  
% 		\begin{align*}
% 			\Bigg\downarrow \textrm{ after relabeling } 
% 		\end{align*}
% 		\begin{align}
% 			\nonumber 
% 			\begin{array}{ccc ccc ccc}
% 				& & a_1 & \cdots & a_{j-1} & a_j & \widetilde{a}_{j+1} & \cdots & \widetilde{a}_s \\ 
% 				& b_1 & \cdots & b_{j-1} & b_j & \widetilde{b}_{j+1} & \cdots & \widetilde{b}_s & \\ 
% 				\textcolor{lightgray}{a_1} & \textcolor{lightgray}{\cdots} & 
% 				\textcolor{lightgray}{a_{j-1}} & \textcolor{lightgray}{a_j} & 
% 				\textcolor{lightgray}{\widetilde{a}_{j+1}} & \textcolor{lightgray}{\cdots} & 
% 				\textcolor{lightgray}{\widetilde{a}_s} & & 
% 			\end{array}
% 		\end{align}
		At the beginning of the first run, we do not have $a_1$ or $b_1$ in the cylindric partition, 
		but rather $\widetilde{a}_1$ or $\widetilde{b}_1$, respectively.  
		However, at the end of each run, the leftmost so many parts in the first and the second 
		rows of the cylindric partition are labeled $a_1$, $b_1$, etc.  
% 		This is what is depicted either in the steps of the procedure, 
% 		or in \eqref{cylPtnBackwardPhaseStart}.  
		
		Because we deleted $(2j-1)$ from $\beta$, and we added 1 to exactly $(2j-1)$ of 
		the parts in the intermediate cylindric partition, 
		the sum of weights of $\beta$ and of the intermediate cylindric partition remains constant.  
		It equals the sum of weights of $\mu$ and the original $\beta$.  
		
		The relabeling of $(a_1 + 1)$ as $a_1$ etc. does not interfere 
		with any of the operations before it, 
		and certainly not any of the possible operations that some after it; 
		therefore, it should not cause any confusion.  
		
		We tacitly assumed that $j < s$ in the displayed cylindric partition above.  
		This does not have to be the case, as $\beta$ may have a part 
		greater than the length of $\mu$.  
		The remedy is to append zeros, and increase $s$ as much as necessary.  
		This takes care of the extreme case of $\widetilde{\lambda}$ 
		being the empty partition.  
		All of the arguments above apply for non-negative parts as well as strictly positive parts.  
		We also implicitly assumed that $\beta$ is nonempty to start with.  
		If $\beta$ is the empty partition, 
		we do not need the perform any operations $\widetilde{\lambda}$ at all.  
		We simply call $\widetilde{a}_j$'s $a_j$'s, and $\widetilde{b}_j$'s $b_j$'s.  
		
		Once all parts of $\beta$ are exhausted, 
		we clear the trailing pairs of zeros in the cylindric partition at hand, 
		and we declare the obtained cylindric partition $\lambda$.  
		
% 		\begin{align}
% 			\nonumber
% 			\lambda = \begin{array}{ccc ccc}
% 				& & a_1 & a_2 & \ldots & a_r \\ 
% 				& b_1 & b_2 & \ldots & b_r & \\ 
% 				\textcolor{lightgray}{a_1} & \textcolor{lightgray}{a_2} & 
% 				\textcolor{lightgray}{\ldots} & \textcolor{lightgray}{a_r} & & 
% 			\end{array}
% 		\end{align}
		Because the sum of the weights of $\beta$ and of the intermediate cylindric partition 
		remained constant at each step of the transformation 
		and $\beta$ is the empty partition at the end, we have 
		\begin{align}
			\nonumber 
			\vert \lambda \vert = \vert \mu \vert + \vert \textrm{(the original)} \beta \vert.  
		\end{align}
		Except for the relabelings, 
		the adding or subtracting 1's, 
		and adding or deleting parts of $\beta$ 
		are done in exact reverse order, 
		and they are clearly inverse operations of each other, 
		the process is reversible, 
		and the collection of profile $c=(1,1)$ cylindric partitions $\lambda$'s 
		are in one-to-one correspondence with the pairs $(\mu, \beta)$ 
		of an unrestricted partition and a partition into distinct odd parts.  
		The relabelings in the two phases of the proof
		are consistent at the beginning and at the end of the transformation, 
		and between the rounds of operations.  
		This concludes the proof.  
	\end{proof}
	
	The following example demonstrates how we construct the pair of partitions $(\mu, \beta)$ if we are given a cylindric partition $\lambda$ with profile  $c=(1,1)$.  
	\begin{ex} \normalfont
		Let $\lambda$ be the following cylindric partition with profile $c=(1,1)$:
		
		\[
		\begin{array}{ccc ccc ccc}
			& & & 7 & 4 & 4 & 3 &  \\
			& & 6 & 5 & 4 & 
		\end{array}
		\]
		
		We read the parts of $\lambda$ as pairs: $[6,7], [5,4], [4,4]$ and $[0,3]$.
		
		\[
		\begin{array}{ccc ccc ccc}
			& & & 7 & 4 & 4 & 3 &  \\
			& & 6 & 5 & 4 & 0
		\end{array}
		\]
		We now start to perform the moves defined in the proof of Theorem \ref{Fc(1,q) when c=(1,1)}. We first change the places of $0$ and $3$ in the rightmost pair and we get the following intermediate partition:
		
		\[
		\begin{array}{ccc ccc ccc}
			& & & 7 & 4 & 4 & 0 &  \\
			& & 6 & 5 & 4 & \circled{3}
		\end{array}
		\]
		\begin{center}
			$\Big\downarrow \scriptsize\parbox{7cm}{subtract $1$ from circled 3 and  the parts take place above  and on the left of it}$
		\end{center}
		
		\[
		\begin{array}{ccc ccc ccc}
			& & & 6 & 3 & 3 & 0 \\
			& & 5 & 4 & 3 & 2 
		\end{array}
		\]
		
		We changed the total weight by $7$, so we have $\beta_1=7$. We do not touch to the pairs $[3,3]$ and $[4,3]$ since $3 \geq 3$ and $4 \geq 3$. We now correct the places of $6$ and $5$, then we perform the last possible move:
		
		\[
		\begin{array}{ccc ccc ccc}
			& & & 5 & 3 & 3 & 0 \\
			& & \circled{6} & 4 & 3 & 2 
		\end{array}
		\]

		\hspace{70mm}	$\Big\downarrow \scriptsize\parbox{7cm}{subtract $1$ from circled 6 }$
		
		\[
		\begin{array}{ccc ccc ccc}
			& & & 5 & 3 & 3 & 0 \\
			& & 5 & 4 & 3 & 2 
		\end{array}
		\]
		
		We changed the total weight by $1$, so we have $\beta_2=1$. Therefore, we decomposed $\lambda$ into the pair of partitions $(\mu, \beta)$, where $\beta=7+1$ and $\mu=5+5+4+3+3+3+2$.
	\end{ex}
	
	The following example demonstrates how we construct a unique cylindric partition $\lambda$ with profile $c=(1,1)$, if we are given a pair of partitions 
	$(\mu,\beta)$ which is described as in the proof of Theorem \ref{Fc(1,q) when c=(1,1)}.
	
	\begin{ex}\normalfont
		Let $\mu=6+5+5+3+1$ and $\beta=9+7+3$. We read the parts of $\mu$ as follows:
		\[
		\begin{array}{ccc ccc ccc}
			& & & 5 & 3 & 0 & 0 & 0 &  \\
			& & 6 & 5 & 1 & 0 & 0 & 
		\end{array}
		\]
		The first part of $\beta$ is $9$. Since we want to increase the weight by $9$, we add $0$'s as many as we need when we construct the pairs.
		
		\[
		\begin{array}{ccc ccc ccc}
			& & & 5 & 3 & 0 & 0 & 0 &  \\
			& & 6 & 5 & 1 & 0 & \circled{0} & 
		\end{array}
		\]
		\hspace{70mm}	$\Big\downarrow \scriptsize\parbox{7cm}{increase by $1$ circled 0 and all corresponding parts }$
		\[
		\begin{array}{ccc ccc ccc}
			& & & 6 & 4 & 1 & 1 & 0 &  \\
			& & 7 & 6 & 2 & 1& 1 & 
		\end{array}
		\]
		\hspace{70mm}	$\Big\downarrow \scriptsize\parbox{7cm}{correct the places of parts in the last pair }$
		\[
		\begin{array}{ccc ccc ccc}
			& & & 6 & 4 & 1 & 1 & 1 &  \\
			& & 7 & 6 & 2 & \circled{1}& 0 & 
		\end{array}
		\]
		
		\hspace{70mm}	$\Big\downarrow \scriptsize\parbox{7cm}{increase by $1$ circled 0 and all corresponding parts  }$
		\[
		\begin{array}{ccc ccc ccc}
			& & & 7 & 5 & 2 & 1 & 1 &  \\
			& & 8 & 7 & 3 & 2& 0 & 
		\end{array}
		\]
		\hspace{70mm}	$\Big\downarrow \scriptsize\parbox{7cm}{correct the places of parts in the second pair from the right }$
		
		\[
		\begin{array}{ccc ccc ccc}
			& & & 7 & 5 & 2 & 2 & 1 &  \\
			& & 8 & \circled{7} & 3 & 1& 0 & 
		\end{array}
		\]
		\hspace{70mm}	$\Big\downarrow \scriptsize\parbox{7cm}{increase by $1$ circled 7 and all corresponding parts  }$
		\[
		\begin{array}{ccc ccc ccc}
			& & & 8 & 5 & 2 & 2 & 1 &  \\
			& &9 & 8 &  3 & 1& 0 & 
		\end{array}
		\]
		\hspace{70mm}	$\Big\downarrow \scriptsize\parbox{7cm}{correct the places of parts in the fourth pair from the right }$
		\[
		\begin{array}{ccc ccc ccc}
			\lambda=	& & & 8 & 8 & 2 & 2 & 1 &  \\
			& &9 & 5 &  3 & 1& 0 & 
		\end{array}
		\]
		$\lambda$ is the unique cylindric partition with profile $c=(1,1)$ corresponding to the pair of partitions $(\mu,\beta)$.
	\end{ex}
	
	\begin{theorem} 
		Let $c=(1,1)$. Then the generating function of cylindric partitions with profile $c$ is given by 
		
		\begin{equation*} 
			F_c(z,q) = \frac{(-zq;q^2)_\infty}{(zq;q)_\infty}.
		\end{equation*}
		where the exponent of variable $z$ keeps track of the largest part of the cylindric partitions.
		
		\begin{proof}
			In the proof of Theorem \ref{Fc(1,q) when c=(1,1)}, we show that there is a one-to-one correspondence between the cylindric partitions with profile $c=(1,1)$ and the pairs of partitions $(\mu,\beta)$ such that $\mu$ is an ordinary partition and $\beta$ is a partition into distinct odd parts. For the proof, we will use this correspondence. 
			If we take a pair of partitions $(\mu,\beta)$, then during the construction of $\lambda$, each part of $\beta$ increases the largest part of $\mu$ by $1$. Hence, when the whole procedure is done, the largest part of $\mu$ is increased by the number of parts in $\beta$. Because of that fact, we write the generating function of $\beta$ by keeping track of the number of parts, which gives $(-zq;q^2)_\infty$.
			
			The partition $\mu$ is an ordinary partition and the generating function of ordinary partitions such that the largest part is $M$ is given by 
			\begin{align*}
				\frac{q^M}{(1-q)\ldots(1-q^M)}.
			\end{align*}
			If we take sum over all $M$ by keeping track of the largest  parts with the exponent of $z$, we get
			\begin{align*}
				\sum_{M\geq0}\frac{z^Mq^M}{(1-q)\ldots(1-q^M)}=\sum_{M\geq0}\frac{(zq)^M}{(q;q)_M}=\frac{1}{(zq;q)_\infty}.
			\end{align*}
			The second identity follows from Euler's identity~\cite{TheBlueBook}. There is a one-to-one correspondence 
			between the partitions with exactly $k$ parts 
			and the partitions with largest part equals to $k$ via conjugation \cite{TheBlueBook}. 
			Thus, the latter generating function can also be considered as the generating function of ordinary partitions,	where the exponent of $z$ keeps track of the number of parts. Finally, since $\mu$ and $\beta$ are two independent partitions, we get the desired generating function. 
			
		\end{proof}
	\end{theorem}
	\begin{theorem} \label{Fc(1,q) when c=(2,0)}
		Let $c=(2,0)$. Then the generating function of cylindric partitions with profile $c$ is given by 
		
		\begin{equation*} 
			F_c(1,q) = \frac{(-q^2;q^2)_\infty}{(q; q)_\infty}.
		\end{equation*}
	\end{theorem}
	
	\begin{proof}
		The proof is very similar to the proof of Theorem \ref{Fc(1,q) when c=(1,1)}. 
		We read the parts of the cylindric partition

		\begin{align}
			\nonumber
			\lambda = \begin{array}{ccc ccc ccc}
				& & a_0 & a_1 & a_2 & \ldots & a_{r-1} & a_r \\
				& &b_1 & b_2  & \ldots & b_{r-1} & b_{r} \\ 
				\textcolor{lightgray}{a_0} & \textcolor{lightgray}{a_1} & \textcolor{lightgray}{a_2} & \textcolor{lightgray}{a_3} 
				& \textcolor{lightgray}{\ldots} & \textcolor{lightgray}{a_r} & 
			\end{array}.  
		\end{align}
		
		as the pairs: $[b_1,a_1], [b_2,a_2], [b_3,a_3], \ldots, [b_r,a_r]$. We note that the largest part of the cylindric partition $\lambda$, namely, $a_0$ is not contained in any pairs. We consider it as a single part.	
		
		$a_0$ is not switched with any part, 
		but it is increased or decreased accordingly 
		when we construct or incorporate $\beta$ 
		as in the proof of Theorem \ref{Fc(1,q) when c=(1,1)}.  
		Thus, $\beta$ consists of distinct even parts, 
		as opposed to distinct odd parts.  
	\end{proof}
	
	If we construct the generating function of cylindric partitions with profile $c=(2,0)$ by using \eqref{BorodinProd}, we get 
	\begin{equation*} 
		F_c(1,q) = \frac{1}{(q;q)_\infty(q^2;q^4)_\infty}.
	\end{equation*}
	
	If we compare that generating function with the generating function in Theorem \ref{Fc(1,q) when c=(2,0)}, we see that they are equal. Both generating functions have the factor $(q;q)_\infty$ in the denominators. If we cancel that factor, we should check whether 
	
	\begin{equation} \label{Borodin-check}
		\frac{1}{(q^2;q^4)_\infty}=(-q^2;q^2)_\infty
	\end{equation}
	
	or not. This identity holds due to the beautiful identity of Euler which states that the number of partitions of a non-negative integer $n$ into odd parts is equal to the number of partitions of $n$ into distinct parts. To obtain \eqref{Borodin-check}, we make the substitution $q^2 \rightarrow q$ in Euler's identity. 
	
	\section{Cylindric partitions into distinct parts}
	
	If all parts in a cylindric partition with profile $c=(1,1)$ are distinct, 
	then the inequalities between parts are strict.  
	Given such a partition, if we label the parts in the top row as $a_1, a_2, \ldots$, 
	and the parts in the bottom row as $b_1, b_2, \ldots$, 
	\begin{align}
		\label{ptnC11CylDistPartGeneric}
		\begin{array}{ccc ccc ccc}
			&  & a_1 & a_2 & \cdots & a_{r-1} & a_r & \cdots & a_n \\
			& b_1 & b_2 & \cdots & b_{r-1} & b_r & \cdots & b_n & \\
			\textcolor{lightgray}{a_1} & \textcolor{lightgray}{a_2} & \textcolor{lightgray}\cdots 
			& \textcolor{lightgray}{a_{r-1}} & \textcolor{lightgray}{a_r} 
			& \textcolor{lightgray}\cdots & \textcolor{lightgray}{a_n} & & 
		\end{array}, 
	\end{align}
	we have the inequalities 
	\begin{align}
		\nonumber
		a_r > a_{r+1}, 
		\qquad 
		b_r > b_{r+1}, 
		\qquad 
		a_r > b_{r+1}, 
		\qquad 
		\textrm{ and }
		\qquad 
		b_r > a_{r+1}
	\end{align}
	for $r = 1, 2, \ldots, n-1$.  
	In particular,
	\begin{align}
		\label{ineqDistC11}
		\mathrm{min}\{ a_r, b_r \} > \mathrm{max}\{ a_{r+1}, b_{r+1} \}
	\end{align}
	for $r = 1, 2, \ldots, n-1$.  
	As in the proof of Theorem \ref{Fc(1,q) when c=(1,1)}, 
	we lose no generality by assuming that the top row and the bottom row have equal number of parts.  
	We achieve this by allowing one of $a_n$ or $b_n$ to be zero.  
	The inequality \eqref{ineqDistC11} ensures that we can switch the places of $a_r$ and $b_r$ 
	without violating the condition for cylindric partition with profile $c=(1, 1)$ 
	for $r = 1, 2, \ldots, n$.  
	There are $2^n$ ways to do this.  
	
	Therefore, given a cylindric partition into $2n$ distinct parts with profile $c=(1, 1)$, 
	we can switch places of $a_r$ and $b_r$ to make $b_r > a_r$ for $r = 1, 2, \ldots, n$ so that 
	\begin{align}
		\label{ineqDistC11aug}
		b_1 > a_1 > b_2 > a_2 > \cdots > b_n > a_n, 
	\end{align}
	where $a_n$ is possibly zero.  
	In other words, we obtain a partition into $2n$ distinct parts 
	in which the smallest part is allowed to be zero.  
	
	Conversely, if we start with a partition into $2n$ distinct parts 
	in which the smallest part can be zero, 
	we can label the parts as in \eqref{ineqDistC11aug}, 
	then place them as in \eqref{ptnC11CylDistPartGeneric}, 
	and allow switching places of $a_r$ and $b_r$ for $r = 1, 2, \ldots, n$; 
	then we will have generated a cylindric partition into $2n$ distinct parts with profile $c=(1,1)$, 
	where one of the parts is allowed to be zero.  
	
	It is clear that any such cylindric partition corresponds to a unique 
	partition into an even number of distinct parts, 
	and any partition into $2n$ distinct parts gives rise to $2^n$ cylindric partitions.  
	We have almost proved the following lemma.  
	
	\begin{lemma}
		\label{lemmaGenFuncCylPtnC11Dist}
		Let $d_{(1,1)}(m, n)$ denote the number of cylindric partitions of $n$ 
		into distinct parts with profile $c = (1, 1)$ and the largest part equal to $m$.  
		Then, 
		\begin{align}
			\nonumber 
			D_{(1,1)}(t, q) 
			= \sum_{n, m \geq 0} d_{(1,1)}(m, n) t^m q^n 
			= \sum_{n \geq 0} \frac{ q^{\binom{2n}{2}} t^{2n-1} 2^n }{ (tq; q)_{2n} }.  
		\end{align}
	\end{lemma}
	\begin{proof}
		We build the proof on the discussion preceding the statement of the lemma.  
		A partition into $2n$ distinct parts is generated by 
		\begin{align}
			\nonumber 
			\frac{q^{\binom{2n}{2}}}{ (q; q)_{2n} }, 
		\end{align}
		where the smallest part is allowed to be zero.  
		When we want to keep track of the largest part, 
		we start by the minimal partition into $2n$ distinct parts 
		\begin{align}
			\nonumber 
			(2n-1), (2n-2), \ldots, 1, 0, 
		\end{align} 
		hence the $t^{2n-1}$ in the numerator in the rightmost sum in the Lemma.  
		Then, for $j = 1, 2, \ldots, 2n$, 
		the factor $(1 - tq^j)$ in the denominator contributes to the 
		$j$ largest parts in the partition into distinct parts.  
	\end{proof}
	
	A similar discussion ensues for cylindric partitions into distinct parts 
	with profile $c=(2, 0)$.  
	The generic cylindric partition is 
	\begin{align}
		\label{ptnC20CylDistPartGeneric}
		\begin{array}{ccc ccc ccc c}
			& &  a_0 & a_1 & a_2 & \cdots & a_{r-1} & a_r & \cdots & a_n \\
			& & b_1 & b_2 & \cdots & b_{r-1} & b_r & \cdots & b_n & \\
			\textcolor{lightgray}{a_0} & \textcolor{lightgray}{a_1} & \textcolor{lightgray}{a_2} 
			& \textcolor{lightgray}\cdots 
			& \textcolor{lightgray}{a_{r-1}} & \textcolor{lightgray}{a_r} 
			& \textcolor{lightgray}\cdots & \textcolor{lightgray}{a_n} & & 
		\end{array}, 
	\end{align}
	where the top row contains $n+1$ parts, the bottom row contains $n$ parts, 
	and one of $a_n$ or $b_n$ is allowed to be zero.  
	We still have the inequalities \eqref{ineqDistC11}.  
	In addition, $a_0$ is the absolute largest part.  
	Needless to say that $a_0$ is zero if and only if we have the empty cylindric partition.  
	$a_r$ and $b_r$ can switch places for $r = 1, 2, \ldots, n$ to obtain 
	the augmented chain of inequalities
	\begin{align}
		\nonumber
		% \label{ineqDistC20aug}
		a_0 > b_1 > a_1 > b_2 > a_2 > \cdots > b_n > a_n, 
	\end{align}
	to get a partition into $(2n+1)$ distinct parts, 
	where the smallest part is allowed to be zero.  
	Conversely, any partition into $(2n+1)$ distinct parts 
	in which the smallest part is allowed to be zero 
	gives rise to exactly $2^n$ cylindric partitions into distinct parts with profile $c = (2, 0)$.  
	We have again almost proved the following lemma.  
	\begin{lemma}
		\label{lemmaGenFuncCylPtnC20Dist}
		Let $d_{(2,0)}(m, n)$ denote the number of cylindric partitions of $n$ 
		into distinct parts with profile $c = (2, 0)$ and the largest part equal to $m$.  
		Then, 
		\begin{align}
			\nonumber 
			D_{(2,0)}(t, q) 
			= \sum_{n, m \geq 0} d_{(2,0)}(m, n) t^m q^n 
			= \sum_{n \geq 0} \frac{ q^{\binom{2n+1}{2}} t^{2n} 2^n }{ (tq; q)_{2n+1} }.  
		\end{align}
	\end{lemma}
	The proof is almost the same as that of Lemma \ref{lemmaGenFuncCylPtnC11Dist}, 
	and is skipped.  
	
	Next, we dissect one of Euler's $q$-series identities~\cite{TheBlueBook}
	\begin{align}
		\nonumber
		\sum_{n \geq 0} \frac{ q^{\binom{n}{2}} a^n }{ (q; q)_n } = (-a; q)_\infty
	\end{align}
	to separate odd and even powers of $a$.  
	\begin{align}
		\nonumber
		\sum_{n \geq 0} \frac{ q^{\binom{2n}{2}} (a^2)^n }{ (q; q)_{2n} } 
		= \frac{ (-a; q)_\infty + (a; q)_\infty }{2}, 
		\qquad 
		a \sum_{n \geq 0} \frac{ q^{\binom{2n+1}{2}} (a^2)^n }{ (q; q)_{2n+1} } 
		= \frac{ (-a; q)_\infty - (a; q)_\infty }{2}.  
	\end{align}
	If we plug in $t = 1$ in Lemmas \ref{lemmaGenFuncCylPtnC11Dist} and \ref{lemmaGenFuncCylPtnC20Dist}, 
	and $a = \sqrt{2}$ in the above formulas, we obtain the following theorem.  
	We repeat the descriptions of the partition enumerants for ease of reference.  
	\begin{theorem}
		\label{thmGenFuncCylPtnDistC11C20}
		Let $d_{(1,1)}(m, n)$ and $d_{(2, 0)}(m, n)$ 
		be the number of cylindric partitions of $n$ into distinct parts 
		where the largest part is $m$
		with profiles $c=(1,1)$ and $c=(2,0)$, respectively.  
		Let 
		\begin{align}
			\nonumber
			D_{(1, 1)}(t, q) = \sum_{m, n \geq 0} d_{(1,1)}(m, n) t^m q^n, 
			\quad \textrm{ and } \quad 
			D_{(2, 0)}(t, q) = \sum_{m, n \geq 0} d_{(2,0)}(m, n) t^m q^n 
		\end{align}
		be the respective generating functions.  Then, 
		\begin{align}
			\nonumber 
			D_{(1,1)}(1, q) & = \frac{ (-\sqrt{2}; q)_\infty + ( \sqrt{2}; q)_\infty }{ 2 }, \\ 
			\nonumber 
			D_{(2,0)}(1, q) & = \frac{ (-\sqrt{2}; q)_\infty - ( \sqrt{2}; q)_\infty }{ 2 \sqrt{2} }, \\ 
			\nonumber 
			D_{(1,1)}(1, q) + \sqrt{2} \; D_{(2,0)}(1, q) & = (-\sqrt{2}; q)_\infty.  
		\end{align}
	\end{theorem}
	
	\section{Discussion}
	
	In Theorem \ref{Fc(1,q) when c=(1,1)}, to construct the desired generating function for cylindric partitions with profile $c=(1,1)$, we decompose each  cylindric partition $\lambda$ with profile $c=(1,1)$ into a pair of partitions $(\mu,\beta)$, where $\mu$ is an ordinary partition and $\beta$ is a partition with distinct odd parts. Conversely, if a pair $(\mu,\beta)$ is given, then we find a unique cylindric partition $\lambda$ with profile $c=(1,1)$. In that way, we find a one-to-one correspondence between $\lambda$'s and $(\mu,\beta)$'s. Here, the partitions $\mu$ and $\beta$ are two independent partitions. In the following theorem, we consider the pair of partitions $(\mu,\beta)$ such that 
	$\mu$ and $\beta$ are dependent partitions and we construct a family of cylindric partitions with profile $c=(1,1)$ corresponding to the pairs $(\mu,\beta)$.
	
	\begin{theorem}
		Let $O_c(1,q)$ be the generating function of cylindric partitions with profile $c=(1,1)$ such that all parts are odd. Then,
		\begin{align} \label{O_c(1,q)}
			O_c(1,q)=&\sum_{k \geq 0}\frac{q^{2k}}{(q^2;q^2)_{2k}}.(-q^2;q^4)_{k}+\sum_{k \geq 0}\frac{q^{2k+1}}{(q^2;q^2)_{2k+1}}.(-q^2;q^4)_{k+1} \\
			=&\sum_{k \geq 0}\frac{q^{2k}(-q^2;q^4)_{k}(1+q-q^{4k+2}+q^{4k+3})}{(q^2;q^2)_{2k+1}} \nonumber.
		\end{align}
	\end{theorem}
	
	\begin{proof}
		We will construct a one-to-one correspondence between the cylindric partitions with profile $c=(1,1)$ such that all parts are odd and the pairs of partitions $(\mu,\beta)$, where $\mu$ is a partition into odd parts,  $\beta$ is a partition into distinct odd parts such that each part is counted twice. Moreover, $\mu$ and $\beta$ are dependent on each other with respect to the number of parts in $\mu$ and the largest odd part in $\beta$ as follows:
		
		\begin{enumerate}[(a)]
			\item if $\mu$ has $2k$ parts,  the largest odd part in $\beta$ is $2k-1$, 
			\item if $\mu$ has $2k+1$ parts,  the largest odd part in $\beta$ is $2k+1$.
		\end{enumerate}
		
		By using exactly the same construction in the proof of Theorem \ref{Fc(1,q) when c=(1,1)}, whenever a pair of partitions $(\mu,\beta)$ as in case $(a)$ or $(b)$, we may construct a unique cylindric partition $\lambda$ with profile $c=(1,1)$ such that all parts are odd. The only change in the construction is that we increase/decrease the weight of each part in the intermediate cylindric partition by $2$ instead of $1$. It is clear that the parts of the cylindric partition have to be odd, since all parts of the partition $\mu$ are odd and the parts of $\beta$ are counted twice, i.e., we do not change the parity of the parts during the transformations. Conversely, if we are given a cylindric partition $\lambda$ with profile $c=(1,1)$ such that all parts are odd, then we may find a unique pair of partitions $(\mu,\beta)$ just described as above. The first term in the sum in \eqref{O_c(1,q)} is the generating function of pairs $(\mu,\beta)$ having the property $(a)$ and the latter term in the summation is the generating function of pairs $(\mu,\beta)$ having the property $(b)$.
	\end{proof}
	
	A natural question is to ask if similar constructions to the proof of 
	Theorems \ref{Fc(1,q) when c=(1,1)} and \ref{Fc(1,q) when c=(2,0)} 
	could be done for cylindric partitions with larger profiles.  
	Another natural question is to ask if similar infinite product generating functions
	to Theorem \ref{thmGenFuncCylPtnDistC11C20}
	could be discovered for cylindric partitions with larger profiles into distinct parts.

\section*{Acknowledgements}

The authors are indebted to the anonymous referee for careful scrutinization of the paper, 
for helpful suggestions to improve the exposition, and for pointing out~\cite{BU}.  
The authors also thank Ole Warnaar for notifying them of 
the connections between Section \ref{secGenFuncsUnrestricted} and~\cite{Warnaar}.  
	
	\bibliographystyle{amsplain}
	
\end{document}